\documentclass[
final
 , nomarks
]{dmtcs-episciences}

\usepackage{latexsym}
\usepackage{mathrsfs}
\usepackage{amsmath,amsthm}
\usepackage{graphicx}
\usepackage{amssymb}
\usepackage{CJK}
\usepackage{color}
\usepackage[round]{natbib}
\usepackage{epstopdf}

\newtheorem{thm}{Theorem}[section]
\newtheorem{cor}[thm]{Corollary}
\newtheorem{lem}[thm]{Lemma}
\newtheorem{prop}[thm]{Proposition}
\newtheorem{defn}[thm]{Definition}

\DeclareGraphicsExtensions{.eps,.eps.gz}
\normalsize \rm
\usepackage{geometry}
\usepackage[centerlast]{caption}

\author{Lingjuan Shi \and Heping Zhang \thanks{Corresponding author.}}
\title{Tight upper bound on the maximum anti-forcing numbers of graphs \thanks{This work is fully supported by  National Natural Science Foundation of China (grant no. 11371180) and the Fundamental Research Funds for the Central Universities (grant nos. lzujbky-2017-ct01, lzujbky-2016-ct12).}}

\affiliation{School of Mathematics and Statistics, Lanzhou University, P.R. China}


\received{2017-4-14}
\revised{2017-8-10}
\accepted{2017-9-28}

\begin{document}
\publicationdetails{19}{2017}{3}{9}{3267}
\maketitle
\begin{abstract}
Let $G$ be a simple graph with a perfect matching.
Deng and Zhang showed that the maximum anti-forcing number of $G$ is no more than the cyclomatic number. In this paper, we get a novel upper bound on the maximum anti-forcing number of $G$ and investigate the extremal graphs. If  $G$ has a  perfect matching $M$ whose anti-forcing number attains this upper bound, then we say $G$ is an extremal graph and $M$ is a {\em nice} perfect matching.
We obtain an equivalent condition for the nice perfect matchings of $G$ and establish a one-to-one correspondence between the nice perfect matchings and  the edge-involutions of $G$, which are the automorphisms $\alpha$ of order two such that $v$ and $\alpha(v)$ are adjacent for every vertex $v$.
We demonstrate that all extremal graphs can be constructed from $K_2$ by implementing two expansion operations, and $G$ is extremal if and only if one factor in a Cartesian decomposition of $G$ is extremal.
As examples, we have that all  perfect matchings of the complete graph $K_{2n}$ and the complete bipartite graph $K_{n, n}$ are nice.
Also we show that the hypercube $Q_n$, the  folded hypercube $FQ_n$ ($n\geq4$) and the  enhanced hypercube $Q_{n, k}$ ($0\leq k\leq n-4$) have exactly $n$, $n+1$ and $n+1$ nice perfect matchings respectively.
%\\[2pt]

%\textbf{Keywords:} Maximum anti-forcing number; Perfect matching; Edge-involution; Cartesian product; Hypercube; Folded hypercube.
\end{abstract}
%%-----------------------------------------------------------------------------------------------------------------------------------
\section{Introduction}
Let $G$ be a finite and simple graph with vertex set $V(G)$ and edge set $E(G)$. We denote the number of vertices of $G$ by $v(G)$, and the number of edges by $e(G)$.
For $S\subseteq E(G)$, $G-S$  denotes the subgraph of $G$ with vertex set $V(G)$ and edge set $E(G)\setminus S$.
A \emph{perfect matching} of $G$ is a set $M$ of edges of  $G$ such that each vertex is incident with exactly one edge of $M$. A perfect matching of a graph coincides with a Kekul\'{e} structure in organic chemistry.

The \emph{innate degree of freedom} of a Kekul\'{e} structure was firstly proposed by \cite{D.J.Klein.1987} in the study of  resonance structure of a given molecule in chemistry. In general, \cite{F.Harary1991} called the innate degree of freedom as the forcing number of a perfect matching of a graph. The \emph{forcing number} of a perfect matching $M$ of a graph $G$ is the smallest cardinality of  subsets of $M$ not contained in other perfect matchings of $G$. The \emph{minimum forcing number} and \emph{maximum forcing number} of $G$ are the minimum and maximum values of forcing numbers over all perfect matchings of $G$, respectively. Computing the minimum forcing number of a bipartite graph with the maximum degree three is an $NP$-complete problem, see \cite{P.Afshani}.
As we know, the forcing numbers of perfect matchings have been studied for many specific graphs, see \cite{ PA, A_survay, XJ, XJ2016, F.Lam, LP1998, SZ2016, ZD2015, HZ, ZZL, ZZ}.

\cite{DV2007} defined the anti-forcing number of a graph as the smallest number of edges whose removal results in a subgraph with a unique perfect matching. Recently \cite{c'(M)} introduced the anti-forcing number of a single perfect matching $M$ of a graph $G$ as follows.  A subset $S\subseteq E(G)\setminus M$ is called an \emph{anti-forcing set of $M$} if $G-S$ has a unique perfect matching $M$. The \emph{anti-forcing number} of a perfect matching $M$ is the  smallest cardinality of anti-forcing sets of $M$, denoted by $af(G,M)$. Obviously, the anti-forcing number of $G$ is the minimum value of the anti-forcing numbers over all perfect matchings of $G$. The \emph{maximum anti-forcing number} of $G$ is the maximum value of the anti-forcing numbers over all perfect matchings of $G$, denoted by $Af(G)$.
It is an $NP$-complete problem to determine the anti-forcing number of a perfect matching of a bipartite graph with the
maximum degree four, see \cite{DK2015}. For some progress on this topic, see refs. \cite{VT, A_survay, HD2007, HD2008, DK2015, KDHZ,  KDHZ2017, c'(M), X.Li, SZ2016, QY2015, Q.Zhang2011}.

For a bipartite graph $G$, \cite{M.Riddle} proposed the trailing vertex method to get a lower bound on the forcing numbers of perfect matchings of $G$. Applying this lower bound, the minimum forcing number of some graphs have been obtained. In particular, \cite{M.Riddle} showed that the minimum forcing number of $Q_n$ is $2^{n-2}$ if $n$ is even. However, for odd $n$,  determining the minimum forcing number of $Q_n$ is still an open problem.
For the maximum forcing number of $Q_n$, Alon  proved that for sufficiently large $n$ this number is near to the total number of edges in a perfect matching of $Q_n$ (see \cite{M.Riddle}), but its specific value is still unknown.
Afterwards, \cite{PA} generalized  Alon's result to a $k$-regular bipartite graph and for a hexagonal system, a polyomino graph or a $(4, 6)$-fullerene, \cite{Xu2013, Zhou2016, SWZ2017} showed that its maximum forcing number equals its Clar number, respectivey.
For a graph $G$ with a perfect matching, \cite{c'(M)} connected the anti-forcing number and forcing number of a perfect matching of $G$, and showed that the maximum forcing number of $G$ is no more than $Af(G)$. Particularly, for a hexagonal system $H$, \cite{c'(M)} showed that $Af(H)$ equals the Fries number (see \cite{Fries}) of $H$. Recently, see \cite{SWZ2017},  we also showed that for a $(4, 6)$-fullerene graph $G$, $Af(G)$ equals the Fries number of $G$.

The \emph{cyclomatic number} of a connected graph $G$ is defined as  $r(G)=e(G)-v(G)+1$.
\cite{KDHZ2017} recently obtained that the maximum anti-forcing number of a graph is no more than the
cyclomatic number.
\begin{thm}[\cite{KDHZ2017}]
For  a connected graph $G$  with a perfect matching, $Af(G)\leq r(G)$.
\end{thm}

\cite{KDHZ2017} further showed that the connected graphs with the maximum anti-forcing number attaining the cyclomatic number  are a class of plane bipartite graphs.
In this paper,  we obtain a novel upper bound on the maximum anti-forcing numbers of a graph $G$  as follows.
\begin{thm}\label{bound}
Let $G$ be any simple graph with a perfect matching. Then for any perfect matching $M$ of $G$,
\begin{equation}\label{upper}af(G, M)\leq Af(G)\leq\frac{2e(G)-v(G)}{4}.\end{equation}
\end{thm}

In fact, this upper bound is also tight. By a simple comparison  we immediately get that the upper bound   is better than the previous  upper bound $r(G)$  when $3v(G)<2e(G)+4$. In next sections we shall see that  many non-planar graphs can attain this upper bound, such as complete graphs $K_{2n}$, complete bipartite graphs $K_{n, n}$, hypercubes $Q_n$, etc.

We say that a graph $G$ is {\em extremal} if the maximum anti-forcing number $Af(G)$ attains the upper bound in Theorem \ref{bound}, that is, $G$ has  a perfect matching $M$ such that both equalities in (\ref{upper}) hold. Such $M$ is said to be a \emph{nice} perfect matching of $G$.
In Section $2$, we  give a proof to  Theorem \ref{bound},  obtain an equivalent condition for the nice perfect matchings of $G$, and establish a one-to-one correspondence between the nice perfect matchings of $G$ and the edge-involutions of $G$.
In Section $3$, we provide a construction of all extremal graphs, which  can be obtained from $K_2$ by implementing two expansion operations, and show that such a graph is an elementary graph (each edge belongs to some perfect matching).
In Section $4$, we investigate  Cartesian decompositions of  an extremal graph. Let $\Phi^\ast(G)$ denote the number of nice perfect matchings of a graph $G$.  For a Cartesian decomposition $G=G_1\Box\cdots\Box G_k$, we obtain  $\Phi^\ast(G)=\sum_{i=1}^{k}\Phi^\ast(G_i)$. This implies that a graph $G$ is extremal if and only if in a Cartesian decomposition of $G$ one factor  is an extremal graph. As applications we show that three cube-like graphs, the
hypercubes $Q_n$, the folded hypercubes $FQ_n$  and the  enhanced hypercubes $Q_{n, k}$ are extremal.  In particular, in the final section we prove that $Q_n$ has exactly $n$ nice perfect matchings and $Af(Q_n)=(n-1)2^{n-2}$,  $FQ_n$ ($n\geq4$) has exactly $n+1$ nice perfect matchings and $Af(FQ_n)=n2^{n-2}$, and for $0\leq k\leq n-4$, $Q_{n, k}$ has $n+1$ nice perfect matchings  and $Af(Q_{n, k})=n2^{n-2}$. We also show that $FQ_n$ is a prime graph under the Cartesian decomposition.
%%-----------------------------------------------------------------------------------------------------------------------------------
\section{Upper bound and nice perfect matchings}
\subsection{The proof of Theorem \ref{bound}}
Let $G$ be a graph with a perfect matching $M$. A cycle of $G$ is called an \emph{$M$-alternating cycle} if its edges appear alternately in $M$ and $E(G)\setminus M$. If $G$ has not $M$-alternating cycles, then $M$ is a unique perfect matching since the symmetric difference of two distinct perfect matchings is the union of some $M$-alternating cycles.
So $M$ is a unique perfect matching of $G$ if and only if $G$ has no $M$-alternating cycles. Lei et al. obtained the following characterization for an anti-forcing set of a perfect matching.
\begin{lem}[\cite{c'(M)}]\label{antiforcing set lem}
A set $S\subseteq E(G)\setminus M$ is an anti-forcing set of $M$ if and only if
$S$ contains at least one edge of every $M$-alternating cycle of $G$.
\end{lem}

 A \emph{compatible $M$-alternating set} of $G$ is a set of $M$-alternating cycles such that any two members are either disjoint or intersect only at edges in $M$. Let $c'(M)$ denote the maximum cardinality of compatible $M$-alternating sets of $G$. By Lemma \ref{antiforcing set lem}, the authors obtained the following theorem.
\begin{thm}[\cite{c'(M)}]\label{c'(M)}
For any perfect matching $M$ of $G$, we have $af(G, M)\geq c'(M)$.
\end{thm}
\begin{figure}[htbp!]
\centering
\includegraphics[height=2.3cm]{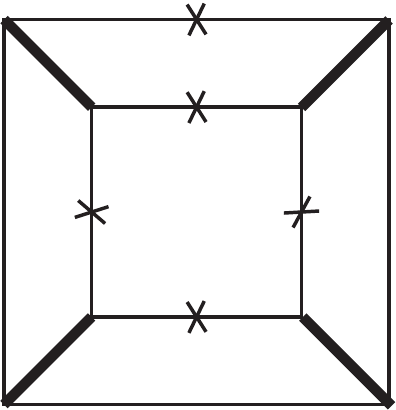}
\caption{\label{Fig.1}{\small A perfect matching $M$ of $Q_3$ (thick edges) and an anti-forcing set $S$ of $M$ (``$\times$").}}
\end{figure}

In general, for any anti-forcing set $S$ of a perfect matching $M$ of $G$, the edge set $E(G)\setminus(M\cup S)$ may not  be an anti-forcing set of $M$ (see Fig. \ref{Fig.1}). However, for any minimal anti-forcing set in a bipartite graph, we have Lemma \ref{complementary set}. Here an anti-forcing set is \emph{minimal} if its any proper subset is not an anti-forcing set. Recall that for an edge subset $E$ of a graph $G$, $G[E]$ is an edge induced subgraph of $G$ with vertex set being the vertices incident with some edge of $E$ and edge set being $E$.
\begin{lem}\label{complementary set}
Let $G$ be a simple bipartite graph with a perfect matching $M$, and $S$ a minimal anti-forcing set of $M$. Then $S^{\ast}:=E(G)\setminus(M\cup S)$ is an anti-forcing set of $M$.
\end{lem}
\begin{proof}
\begin{figure}[htbp!]
\centering
\includegraphics[height=3.5cm]{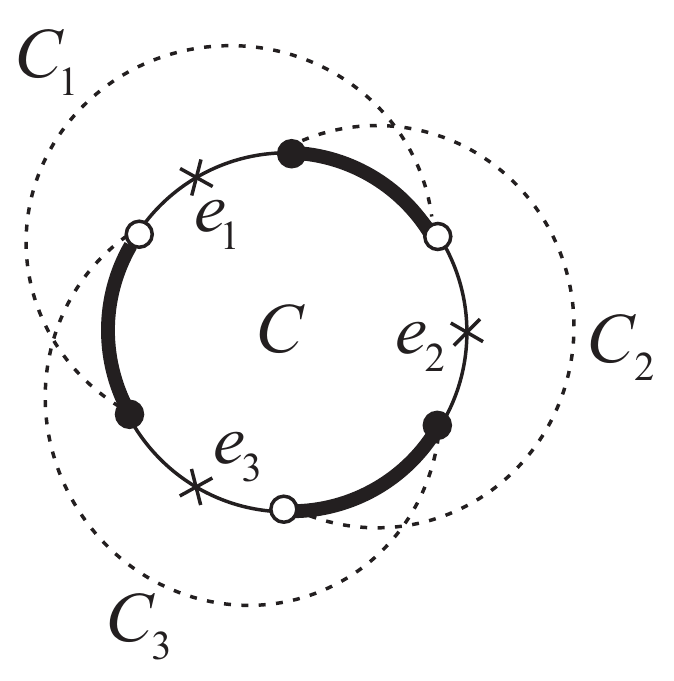}
\caption{\label{Fig.2}{\small Example of $k=3$.}}
\end{figure}
Clearly, $M$ is a perfect matching of $G[M\cup S]$.
It is sufficient to show that $G[M\cup S]$ has no $M$-alternating cycle by Lemma \ref{antiforcing set lem}. By the contrary, we suppose that $C$ is an $M$-alternating cycle of $G[M\cup S]$.
Then the edges of $C$ appear alternately in $M$ and $S$.
Let $E(C)\cap S=\{e_1, e_2, \ldots, e_k\}$ (see Fig. \ref{Fig.2} for $k=3$). Since $S$ is a minimal anti-forcing set of $M$ in $G$, the subgraph $G-(S\setminus \{e_i\})$ has an $M$-alternating cycle $C_i$ such that $E(C_i)\cap S=\{e_i\}$, $i=1, 2, \ldots, k$. Then $G-S$ has a closed $M$-alternating walk $W=G[\bigcup\limits_{i=1}^{k}(E(C_i)\setminus \{e_i\})$ as depicted in Fig. \ref{Fig.2}. Since $G$ is a bipartite graph, $W$ contains an $M$-alternating cycle $C'$. So $G-S$ has an $M$-alternating cycle $C'$. This implies that $S$ is not an anti-forcing set of $M$, a contradiction. So $S^{\ast}$ is an anti-forcing set of $M$.
\end{proof}

Let $X$ and $Y$ be two vertex subsets of a graph $G$. We denote by $E(X, Y)$ the set of edges of $G$ with one end in $X$ and the other end in $Y$. The subgraph induced by $E(X, Y)$, for convenience,  is denoted by $G(X, Y)$.
For a vertex subset $X$ of $G$, $G[X]$ is a vertex induced subgraph of $G$ with vertex set $X$ and any two vertices are adjacent if and only if they are adjacent in $G$. The edge set of $G[X]$ is denoted by $E(X)$.
\\

\noindent{\em Proof of Theorem \emph{\ref{bound}}.}
For any perfect matching $M$ of $G$, let $A$ be a vertex subset of $G$ consisting of one end vertex for each edge of $M$, $\bar{A}:=V(G)\setminus A$. Then $G':=G(A, \bar{A})$ is a bipartite graph and $M$ is a perfect matching of $G'$. Let $S$ be a minimum anti-forcing set of $M$ in $G'$.
By Lemma \ref{complementary set}, $S^{\ast}:=E(G')\setminus(M\cup S)$ is an anti-forcing set of $M$ in $G'$. So both $S\cup E(A)$ and $S^{\ast}\cup E(\bar{A})$ are anti-forcing sets of $M$ in $G$. Hence
$$2af(G, M)\leq|S\cup E(A)|+|S^{\ast}\cup E(\bar{A})|=e(G)-|M|=e(G)-\frac{v(G)}{2}.$$
Then $af(G, M)\leq \frac{2e(G)-v(G)}{4}$. By the arbitrariness of $M$, $Af(G)\leq \frac{2e(G)-v(G)}{4}$. \hfill$\Box$
\\

For any perfect matching $M$ of a complete bipartite graph $K_{m, m}$ ($m\geq2$), any two edges of $M$ belong to an $M$-alternating $4$-cycle. Since any two distinct $M$-alternating $4$-cycles are compatible,  $c'(M)\geq\binom{m}{2}=\frac{m^2-m}{2}$. By Theorems \ref{c'(M)} and \ref{bound}, we obtain $af(K_{m,m}, M)=\frac{m^2-m}{2}=Af(K_{m,m})$.
Let $M'$ be any perfect matching of a complete graph $K_{2n}$. For any two edges $e_1$ and $e_2$ of $M'$, there are two distinct $M'$-alternating $4$-cycles each of which simultaneously contains edges $e_1$ and $e_2$.
So $af(K_{2n}, M')\geq c'(M')\geq \binom{n}{2}\times2=n^2-n$.  By Theorem \ref{bound}, we know that $af(K_{2n}, M')=Af(K_{2n})=n^2-n$. Hence every perfect matching of $K_{m, m}$ and $K_{2n}$ is nice.

Recall that the \emph{$n$-dimensional hypercube} $Q_n$ is the graph with vertex set being the set of all $0$-$1$ sequences of length $n$ and two vertices are adjacent if and only if they differ in exactly one position. For $x\in \{0, 1\}$, set $\bar{x}:=1-x$. The edge connecting the two vertices $x_1\cdots x_{i-1}x_ix_{i+1}\cdots x_n$ and $x_1\cdots x_{i-1}\bar{x}_ix_{i+1}\cdots x_n$ of $Q_n$ is called an \emph{$i$-edge} of $Q_n$. We denote by $E_i$ the set of all the $i$-edges of $Q_n$, $i=1, 2, \ldots, n$.
In fact, $E_i$ is a $\Theta_{Q_n}$-class of $Q_n$.
We can show the following result for $Q_n$.
\begin{lem}\label{Af(Q_n)}
$\Theta_{Q_n}$-class $E_i$ of $Q_n$ is a nice perfect matching, that is, $af(Q_n, E_i)=Af(Q_n)=(n-1)2^{n-2}$.
\end{lem}
\begin{proof}
It is sufficient to discuss $E_1$.
Clearly, $E_1$ is a perfect matching of $Q_n$. For vertices $x=x_1x_2\cdots x_n$ and $y=\bar{x}_1x_2\cdots x_n$, the edge $xy\in E_1$ belongs to $n-1$ $E_1$-alternating $4$-cycles. Over all edges of $E_1$, since each $E_1$-alternating $4$-cycle is counted twice, there are $\frac{(n-1)2^{n-1}}{2}=(n-1)2^{n-2}$ distinct $E_1$-alternating $4$-cycles in $Q_n$. Since any two distinct $E_1$-alternating $4$-cycles are compatible, $c'(E_1)\geq(n-1)2^{n-2}$. So
$af(Q_n, E_1)\geq c'(E_1)\geq(n-1)2^{n-2}$ by Theorem \ref{c'(M)}. Since $Af(Q_n)\leq(n-1)2^{n-2}$ by Theorem \ref{bound}, $af(Q_n, E_1)=Af(Q_n)=(n-1)2^{n-2}$.
\end{proof}

The above three examples show that the upper bound in Theorem \ref{bound} is tight.
\subsection{Nice perfect matchings}
In the following, we will characterize the nice perfect matchings of a graph.
The set of neighbors of a vertex $v$ in $G$ is denoted by $N_G(v)$. The degree of a vertex $v$ is the cardinality of $N_G(v)$, denoted by $d_G(v)$.

\begin{thm}\label{sufficiency and necessary}
For any perfect matching $M$ of a simple graph $G$,  $M$ is a nice perfect matching of $G$ if and only if for any two edges $e_1=xy$ and $e_2=uv$ of $M$, $xu\in E(G)$ if and only if $yv\in E(G)$, and $xv\in E(G)$ if and only if $yu\in E(G)$.
\end{thm}
\begin{proof}
Here we only need to consider simple connected graphs. To show the sufficiency, we firstly estimate the value of $c'(M)$ for such perfect matching $M$ of $G$. Let $c_{wz}'(M)$ be the number of $M$-alternating $4$-cycles that contain edge $wz$. Since for any two edges $e_1=xy$ and $e_2=uv$ of $M$, $xu\in E(G)$ if and only if $yv\in E(G)$, and $xv\in E(G)$ if and only if $yu\in E(G)$, $c_{wz}'(M)=d_G(w)-1=d_G(z)-1$ for any edge $wz$ of $M$.
Obviously, any two distinct $M$-alternating $4$-cycles are compatible. Then
\begin{equation}
\begin{split}
c'(M) &\geq \frac{\sum\limits_{wz\in M}c_{wz}'(M)}{2} \\
&= \frac{\sum\limits_{wz\in M}\frac{1}{2}[(d_G(w)-1)+(d_G(z)-1)]}{2} \\
%&= \frac{\sum\limits_{wz\in M}[\frac{1}{2}(d_G(w)-1)+\frac{1}{2}(d_G(z)-1)]}{2}  \\
&= \frac{\sum\limits_{w\in V(G)}\frac{1}{2}(d_G(w)-1)}{2}\\
&= \frac{e(G)-\frac{v(G)}{2}}{2}. \\
\end{split}
\end{equation}

By Theorems \ref{bound} and \ref{c'(M)}, $c'(M)\leq af(G, M)\leq Af(G)\leq \frac{2e(G)-v(G)}{4}$. So $af(G, M)= \frac{2e(G)-v(G)}{4}$, that is, $M$ is a nice perfect matching of $G$.

Conversely, suppose that $M$ is a nice perfect matching of $G$. Let $A$ be a vertex subset of $G$ consisting of one end vertex for each edge of $M$ and $\bar{A}:=V(G)\setminus A$. Then $(A, \bar{A})$ is a partition of $V(G)$.
Given any bijection $\omega: M\rightarrow \{1, \ldots, |M|\}$, we extend weight function $\omega$ on $M$ to the vertices of $G$: if $v\in V(G)$ is incident with $e\in M$, then $\omega(v):=\omega(e)$. This weight function $\omega$ gives a natural ordering of  the vertices in $A$ $(\bar{A})$.  Clearly, if $e=xy\in M$, then $\omega(x)=\omega(y)$, otherwise, $\omega(x)\neq\omega(y)$. Set
$$E_A^\omega:=\{xy\in E(G): \omega(x)>\omega(y), x\in A~\text{and}~y\in \bar{A}\},$$
$$E_{\bar{A}}^\omega:=\{xy\in E(G): \omega(x)<\omega(y), x\in A~\text{and}~y\in \bar{A}\}.$$

Since $G-E_A^\omega\cup E(A)$ has a unique perfect matching $M$, $E_A^\omega\cup E(A)$ is an anti-forcing set of $M$ in $G$.  Similarly, $E_{\bar{A}}^\omega\cup E(\bar{A})$ is also an anti-forcing set of $M$ in $G$. Since $M$ is a nice perfect matching of $G$, $af(G, M)=\frac{2e(G)-v(G)}{4}$. So $|E_A^\omega\cup E(A)|\geq\frac{2e(G)-v(G)}{4}$, $|E_{\bar{A}}^\omega\cup E(\bar{A})|\geq\frac{2e(G)-v(G)}{4}$. Since $|E_A^\omega\cup E(A)|+|E_{\bar{A}}^\omega\cup E(\bar{A})|=e(G)-|M|=e(G)-\frac{v(G)}{2}$, $|E_A^\omega\cup E(A)|=|E_{\bar{A}}^\omega\cup E(\bar{A})|=\frac{2e(G)-v(G)}{4}$.
Hence $E_A^\omega\cup E(A)$ is a minimum anti-forcing set of $M$ in $G$.

Now we show that for any two edges $e_1=xy$ and $e_2=uv$ of $M$,  $xu\in E(G)$ if and only if $yv\in E(G)$, and $xv\in E(G)$ if and only if $yu\in E(G)$. It is sufficient to show that $xv\in E(G)$ implies $yu\in E(G)$.
Given two bijections $\omega_1: M\rightarrow \{1, \ldots, |M|\}$ and $\omega_2: M\rightarrow \{1, \ldots, |M|\}$ with $\omega_1(e_1)=1$, $\omega_1(e_2)=2$, $\omega_2(e_1)=2$, $\omega_2(e_2)=1$ and $\omega_2|_{M\setminus\{e_1,~ e_2\}}=\omega_1|_{M\setminus\{e_1,~e_2\}}$. As the above extension of $\omega$, we extend the weight functions
$\omega_1$ and $\omega_2$ on $M$ to the vertices of $G$.

We first consider the case that $x,u\in A$. Suppose to the contrary that $xv\in E(G)$ but $yu\notin E(G)$.
Set $A':=A\setminus\{x, u\}$, $\bar{A'}:=\bar{A}\setminus\{y, v\}$,
$E_1':=\{wz\in E(G):\omega_1(w)>\omega_1(z), w\in A'~\text{and}~z\in \bar{A'}\}$. Then
\begin{equation}\label{2}
E_A^{\omega_2}\cup E(A)=\{xv\}\cup E(\{y, v\}, A')\cup E_1'\cup E(A)=\{xv\}\cup E_A^{\omega_1}\cup E(A).\\
\end{equation}
By the above proof we know that both $E_A^{\omega_1}\cup E(A)$ and $E_A^{\omega_2}\cup E(A)$ are minimum anti-forcing sets of $M$ in $G$, it contradicts to the equation (\ref{2}). Thus $yu\in E(G)$.

For the case that $x\in A$ and $u\in \bar{A}$, set $U:=(A\setminus\{v\})\cup\{u\}$, $\bar{U}:=(\bar{A}\setminus\{u\})\cup\{v\}$. Then each edge in $M$ is incident with exactly one vertex in $U$. Substituting the
partition $(A, \bar{A})$ of $V(G)$ with the partition $(U, \bar{U})$, by a similar argument as the above case, we can also show that $xv\in E(G)$ implies $yu\in E(G)$.
%, otherwise, $E_U^{\omega_2}\cup E(U)=\{xv\}\cup E_U^{\omega_1}\cup E(U)$, a contradiction.
\end{proof}
\begin{figure}[htbp!]
\centering
\includegraphics[height=2.65cm]{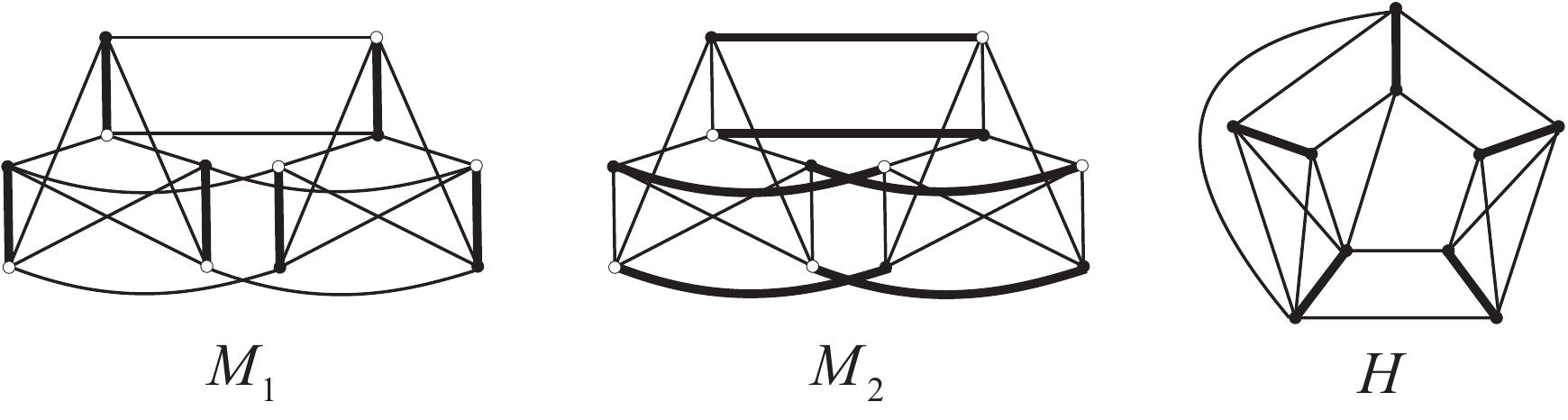}
\caption{\label{Fig.3}{\small Two nice perfect matchings $M_1$ and $M_2$ of $G'$ and a nice perfect matching of $H$.}}
\end{figure}

By Theorem \ref{sufficiency and necessary}, we can easily check whether a perfect matching of a graph is nice.
For example, in Fig. \ref{Fig.3}, the two perfect matchings $M_1$ and $M_2$ of the bipartite graph $G'$ are nice,  and the perfect matching of the non-bipartite graph $H$ is also nice.

\begin{prop}\label{succession}
Let $M$ be a nice perfect matching of $G$ and $S$ a subset of $V(G)$. Then $M\cap E(S)$ is a nice perfect matching of $G[S]$ if $M\cap E(S)$ is a perfect matching of $G[S]$.
\end{prop}
\begin{proof}
By Theorem \ref{sufficiency and necessary}, it holds.
\end{proof}

In the proof of Theorem \ref{sufficiency and necessary}, we notice that $d_G(u)=d_G(v)$ for every edge $uv$ of a nice perfect matching of $G$. So we have the following necessary but not sufficiency condition for the upper bound in Theorem \ref{bound} to be attained.
\begin{prop}\label{necessary_condition}
Let $G$ be a graph with a perfect matching. Then $Af(G)<\frac{2e(G)-v(G)}{4}$ if there are an odd number of vertices of the same degree in $G$.
\end{prop}

Proposition \ref{necessary_condition} is not sufficient. For example,
for a hexagonal system with a perfect matching, it does not have a nice perfect matching by Theorem \ref{sufficiency and necessary}, that is, its maximum anti-forcing number can not be the upper bound in Theorem \ref{bound}, but it has an even number of vertices of degree $3$ and an even number of vertices of degree $2$.

\cite{GA} introduced a \emph{reversing involution} of a connected bipartite graph $G$ with partite sets $X$ and $Y$ as  an automorphism $\alpha$ of $G$ of order two such that $\alpha(X)=Y$ and $\alpha(Y)=X$. Here we give the following definition of a general graph.
\begin{defn}
Suppose that $G$ is a simple connected graph. An edge-involution of $G$ is an automorphism $\alpha$ of $G$ of order two such that $v$ and $\alpha(v)$ are adjacent for any vertex $v$ in $G$.
\end{defn}

Hence an edge-involution of a bipartite graph is also a reversing involution, but a reversing involution of a bipartite graph may not be an edge-involution.
In the following, we establish a relationship between a nice perfect matching and an edge-involution of $G$.
\begin{thm}\label{edge-involution method}
Let $G$ be a simple connected graph. Then there is a one-to-one correspondence between the nice perfect matchings of $G$ and the edge-involutions of $G$.
\end{thm}
\begin{proof}
For a nice perfect matching $M$ of $G$, we define a bijection $\alpha_M$ of order $2$ on $V(G)$ as follows: for any vertex $v$ of $G$, there is exactly one edge $e$ in $M$ such that $v$ is incident with $e$, let $\alpha_M(v)$ be the other end-vertex of $e$. Let $x$ and $y$ be any two distinct vertices of $G$. If $xy\in M$, then $\alpha_M(x)=y$, $\alpha_M(y)=x$ and $\alpha_M(x)\alpha_M(y)=yx\in E(G)$.
If $xy\notin M$ ($x$ may not be adjacent to $y$), then both $x\alpha_M(x)$ and $y\alpha_M(y)$ belong to $M$. Since $M$ is a nice perfect matching, $xy\in E(G)$ if and only if $\alpha_M(x)\alpha_M(y)\in E(G)$ by Theorem \ref{sufficiency and necessary}. This implies that $\alpha_M$ is an automorphism of $G$. Thus $\alpha_M$ is an edge-involution of $G$.

Conversely, let $\alpha$ be an edge-involution of $G$. Then for any vertex $y$ of $G$, $y\alpha(y)\in E(G)$. Since $\alpha$ is a bijection of order $2$ on $V(G)$, $M':=\{y\alpha(y): y\in V(G)\}$ is a perfect matching of $G$. For any two distinct edges $y_1\alpha(y_1)$ and $y_2\alpha(y_2)$ of $M'$, $y_1y_2\in E(G)$ if and only if $\alpha(y_1)\alpha(y_2)\in E(G)$, and $y_1\alpha(y_2)\in E(G)$ if and only if $\alpha(y_1)y_2\in E(G)$ since $\alpha$ is an automorphism of order $2$ of $G$.
So $M'$ is a nice perfect matching of $G$ by Theorem \ref{sufficiency and necessary}. We can also see that $\alpha_{M'}=\alpha$. This establishes a one-to-one correspondence between the nice perfect matchings of $G$ and the edge-involutions of $G$.
\end{proof}

%--------------------------------------------------------------------------------------------------------------------------------
\section{Construction of the extremal graphs}
In the following, we will show that every extremal graph can be constructed from a complete graph $K_2$ by implementing two expansion operations.
\begin{defn}\label{construct defn}
Let $G_i$ be a simple graph with a nice perfect matching $M_i$, $i=1, 2$ \emph{(}note that $V(G_1)\cap V(G_2)=\emptyset$\emph{)}. We define two expansion operations as follows:

$(i)$ $G:=G_i +e +e'$, where $e, e'\notin E(G_i)$ and there are edges $e_1, e_2\in M_i$ such that the four edges $e$, $e'$, $e_1$, $e_2$ form a 4-cycle.

$(ii)$ For $M_1'\subseteq M_1$ and $M_2'\subseteq M_2$ with $|M_1'|=|M_2'|$, given a bijection $\phi$ from $V(M_1')$ to $V(M_2')$ with $uv\in M_1'$ if and only if $\phi(u)\phi(v)\in M_2'$. $G_1$ joins $G_2$ over matchings $M_1'$ and $M_2'$ about bijection $\phi$, denoted by $G_1 \circledast G_2$, is a graph with vertex set $V(G_1)\cup V(G_2)$ and edge set $E(G_1)\cup E(G_2)\cup E'$,  where $E':=\{u\phi(u): u\in V(M_1')\}$.
\end{defn}
\begin{figure}[htbp!]
\centering
\includegraphics[height=4.5cm]{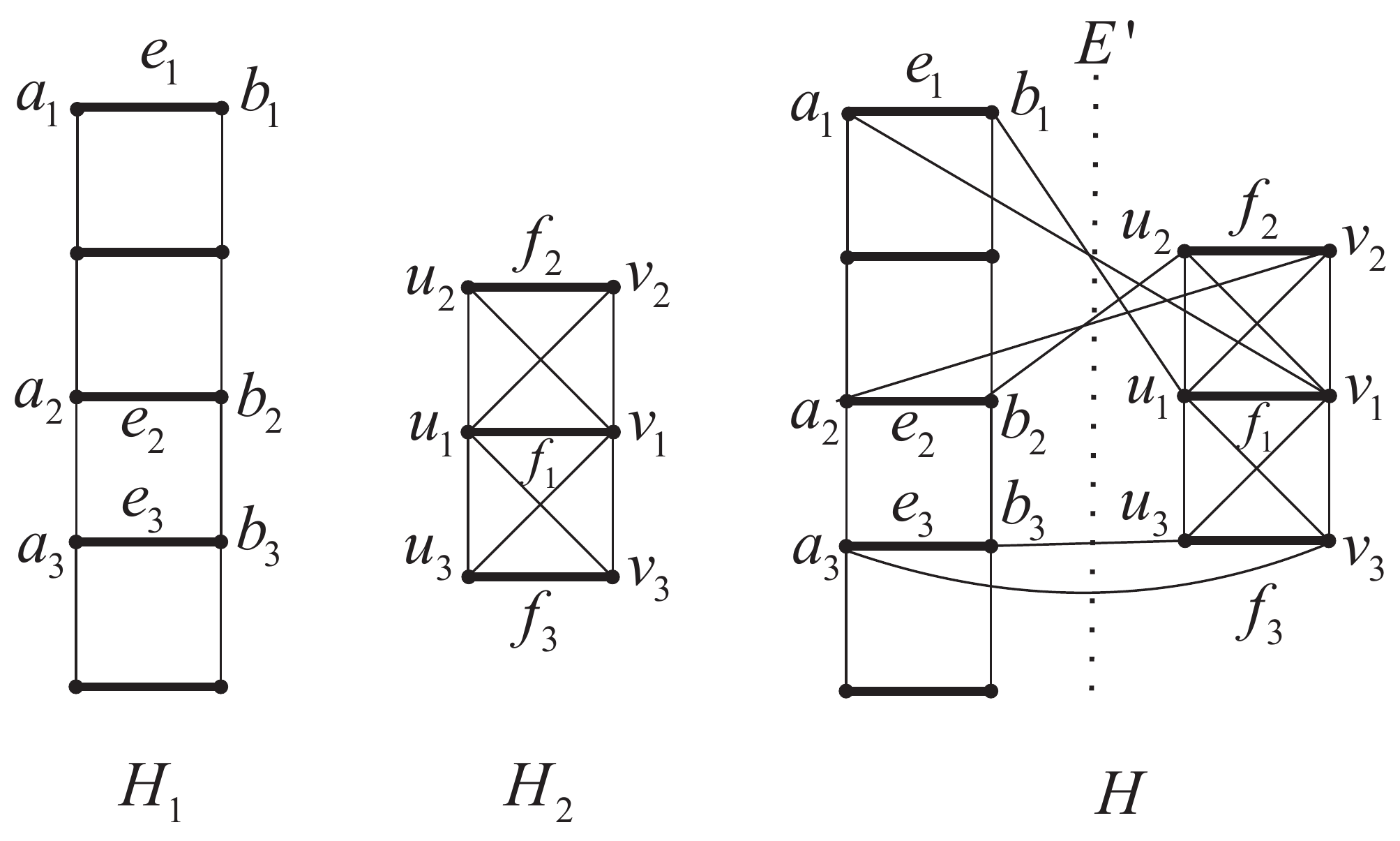}
\caption{\label{Fig.4}{\small $H=H_1\circledast H_2$ over matchings $M_1'$ and $M_2'$ about bijection $\phi'$.}}
\end{figure}

For example, in Fig. \ref{Fig.4} graph $H$ is $H_1\circledast H_2$ over matchings $M_1'$ of $H_1$ and $M_2'$ of $H_2$ about bijection $\phi'$, where $M_1'=\{e_1, e_2, e_3\}$, $M_2'=\{f_1, f_2, f_3\}$, $\phi'(a_i)=v_i$, $\phi'(b_i)=u_i$, $i=1, 2, 3$. $H$ has a nice perfect matching which is marked by thick edges in Fig. \ref{Fig.4}.
Recall that $nK_2$ is the disjoint union of  $n$ copies of $K_2$.
\begin{thm}\label{construct thm}
A simple graph $G$ is an extremal graph if and only if it can be constructed from $K_2$ by implementing operations $(i)$ or $(ii)$ in Definition \emph{\ref{construct defn}} \emph{(}regardless of the orders\emph{)}.
\end{thm}
\begin{proof}
Let $\mathcal{P'}$ be the set of all the graphs that can be constructed from $K_2$ by implementing operations $(i)$ or $(ii)$.
For any graph $G\in\mathcal{P}'$, $G$ is a simple graph with a nice perfect matching by the definition of the two operations.

Conversely, we suppose that $G$ is an extremal graph, that is, $G$ has a nice perfect matching $M=\{e_1, e_2, \ldots, e_n\}$. If $n=1,$ or 2, then $G$ must be isomorphic to $K_2$, $2K_2$, $C_4$ or $K_4$. So $G\in \mathcal{P}'$. Next, we suppose that $n\geq 3$ and it holds for $n-1$. Let $G':=G[\bigcup\limits_{i=1}^{n-1}V(e_i)]$. Then $\{e_1, \ldots, e_{n-1}\}$ is a nice perfect
matching of $G'$ by Proposition \ref{succession}. So $G'\in \mathcal{P}'$ by the induction. If $e_n$ is an isolated edge in $G$, then $G=G'\cup\{e_n\}\in \mathcal{P}'$. Otherwise, $e_n=u_nv_n$ has adjacent edges $u_nv_i$ and $v_nu_i$, or $u_nu_i$ and $v_nv_i$ for some $i\in \{1, \ldots, n-1\}$, where $u_iv_i=e_i\in M$.
Let $G''=G' \circledast K_2$ over matchings $\{e_i\}$ and $\{e_n\}$ about bijection $\phi: \{u_i, v_i\}\rightarrow \{u_n, v_n\}$. So $G''\in \mathcal{P}'$. Then $G$ can  be constructed from $G''$ by implementing several times operations $(i)$. So $G\in \mathcal{P}'$.
\end{proof}
\begin{figure}[htbp!]
\centering
\includegraphics[height=6.8cm]{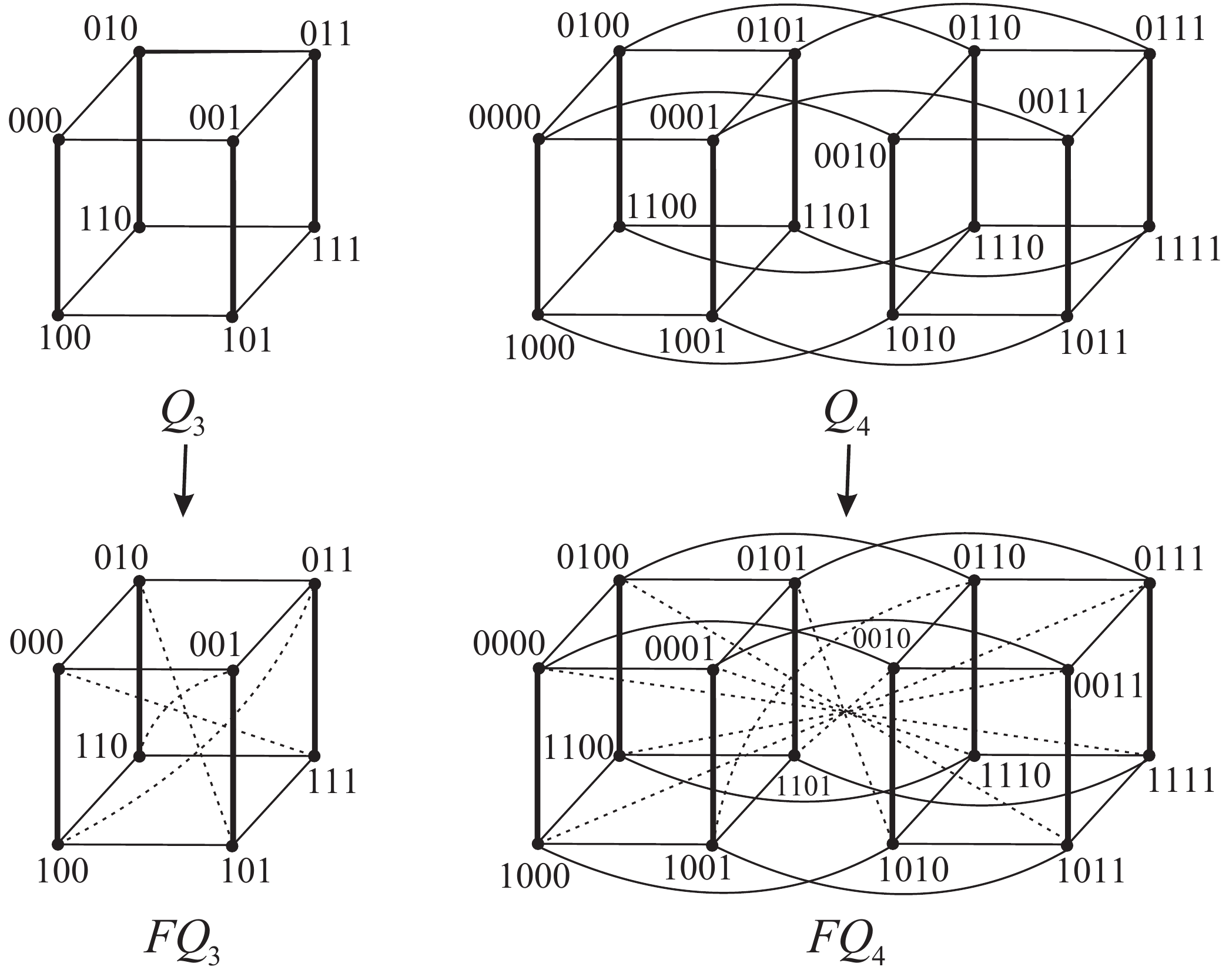}
\caption{\label{Fig.5}{\small The nice perfect matchings $E_1$ of $Q_3$ and $Q_4$ are depicted by thick edges; the dashed edges are the complementary edges.}}
\end{figure}

As a variant of the $n$-dimensional hypercube $Q_n$, the $n$-dimensional \emph{folded hypercube} $FQ_n$, proposed first by \cite{IEEE}, is a graph with $V(FQ_n)=V(Q_n)$ and $E(FQ_n)=E(Q_n)\cup\bar{E}$, where $\bar{E}:=\{x\bar{x}: x=x_1x_2\cdots x_n, \bar{x}=\bar{x}_1\bar{x}_2\cdots\bar{x}_n \}$, $\bar{x}_i:=1-x_i$. Each edge in $\bar{E}$ is called a \emph{complementary edge}.
The graphs shown in Fig. \ref{Fig.5} are $FQ_3$ and $FQ_4$, respectively.
\begin{cor}\label{FQ_n extremal}
$FQ_n$ is an extremal graph and $Af(FQ_n)=n2^{n-2}$.
\end{cor}
\begin{proof}
By Lemma \ref{Af(Q_n)}, $E_1$ is a nice perfect matching of $Q_n$. $FQ_n$ is constructed from $Q_n$ by applying the operation $(i)$ over the nice perfect matching
$E_1$ of $Q_n$ (see Fig. \ref{Fig.5} for $n=3, 4$). So $E_1$ is also a nice perfect matching of the folded hypercube $FQ_n$.
\end{proof}

For any positive integer $n$, a connected graph $G$ with at least $2n+2$ vertices is said to be \emph{$n$-extendable} if every matching of size $n$ is contained in a perfect matching of $G$.
\begin{prop}\label{1-extendable}
Any connected extremal graph $G$ other than $K_2$ is $1$-extendable.
\end{prop}
\begin{proof}
Since $G$ is an extremal graph, it has a nice perfect matching $M$.
For any edge $uv$ of $E(G)\setminus M$ , there are edges $ux$ and $vy$ of $M$.
By Theorem \ref{sufficiency and necessary}, $xy\in E(G)$. So
$uv$ belongs to an $M$-alternating $4$-cycle $C:=uxyvu$. Then $M\bigtriangleup E(C):=(M\cup E(C))\setminus(M\cap E(C))$ is a perfect matching of $G$ that contains edge $uv$. So $G$ is $1$-extendable.
\end{proof}

By Proposition \ref{1-extendable}, any connected extremal graph except for $K_2$ is $2$-connected.

%----------------------------------------------------------------------------------------------------------------------------
\section{Cartesian decomposition}
The \emph{Cartesian product} $G\Box H$ of two graphs $G$ and $H$ is a graph with vertex set $V(G)\times V(H)=\{(x,u):x\in V(G),u\in V(H)\}$ and two vertices $(x,u)$ and $(y,v)$ are adjacent if and only if $xy\in E(G)$ and $u=v$ or $x=y$ and $uv\in E(H)$.
For a  vertex  $(x_i, v_j)$ of $G\Box H$, the subgraphs of $G\Box H$ induced by the vertex set $\{(x, v_j): x\in V(G)\}$ and the vertex set $\{(x_i, v): v\in V(H)\}$ are called a $G$-layer and an $H$-layer of $G\Box H$, and denoted by $G^{v_j}$ and $H^{x_i}$, respectively.

For any graph $H$, let $E'$ be the set of edges of all $K_2$-layers of $H\Box K_2$. Clearly, $E'$ is a perfect matching of $H\Box K_2$. Define a bijection $\alpha$ on $V(H\Box K_2)$ as follows: for every edge $uv\in E'$, $\alpha(u):=v$ and $\alpha(v):=u$. Then $\alpha$ is an edge-involution of $H\Box K_2$. So $H\Box K_2$ is  an extremal graph by Theorem \ref{edge-involution method}. This fact inspires us to consider the Cartesian product decomposition of an extremal graph.
Let $\Phi^\ast(G)$ be the number of all the nice perfect matchings of a graph $G$.
We have Theorem \ref{cartesian  iff}. Recall that for an edge $e=uv$ of $G$ and an isomorphism $\varphi$ from $G$ to $H$,  $\varphi(e):=\varphi(u)\varphi(v)$.
\begin{thm}\label{cartesian  iff}
Let $G_1$ and $G_2$ be two simple connected graphs. Then $$\Phi^\ast(G_1\Box G_2)=\Phi^\ast(G_1)+\Phi^\ast(G_2).$$
\end{thm}
\begin{proof}
Let $V(G_1)=\{x_1, x_2, \ldots, x_{n_1}\}$ and $V(G_2)=\{v_1, v_2, \ldots, v_{n_2}\}$. Since $\Phi^\ast(K_1)=0$,
we suppose $n_1\geq2$ and $n_2\geq2$.

We define  an isomorphism $\rho_{v_i}$ from $G_1$ to $G_1^{v_i}$ and  an isomorphism $\sigma_{x_j}$ from $G_2$ to $G_2^{x_j}$: $\rho_{v_i}(x):=(x, v_i)$ for any vertex $x$ of $G_1$ and $\sigma_{x_j}(v):=(x_j, v)$ for any vertex $v$ of $G_2$.
For any nice perfect matching $M_i$ of $G_i$, $i=1,2$,  let
\begin{equation}\label{map}\rho(M_1):=\bigcup_{v_i\in V(G_2)}\rho_{v_i}(M_1), ~~~\sigma(M_2):=\bigcup\limits_{x_j\in V(G_1)}\sigma_{x_j}(M_2),
\end{equation}
%where $\rho_{v_i}(M_1):=\{\rho_{v_i}(x_s)\rho_{v_i}(x_t): x_sx_t\in M_1\}$ and $\sigma_{x_j}(M_2):=\{\sigma_{x_j}(v_s)\sigma_{x_j}(v_t): v_sv_t\in M_2\}$.
By Theorem \ref{sufficiency and necessary}, $\rho_{v_i}(M_1)$ is a nice perfect matching of $G_1^{v_i}$ and $\rho(M_1)$ is a nice perfect matching of $G_1\Box G_2$. Similarly, $\sigma(M_2)$ is also a nice perfect matching of $G_1\Box G_2$.

Conversely, since $E(G_1^{v_1})$, $\ldots$, $E(G_1^{v_{n_2}})$, $E(G_2^{x_1}),$ $\ldots$, $E(G_2^{x_{n_1}})$ is a partition of $E(G_1\Box G_2)$, for any nice perfect matching $M$ of $G_1\Box G_2$ there is some $x_i$ or $v_j$ such that $M\cap E(G_2^{x_i})\neq\emptyset$ or $M\cap E(G_1^{v_j})\neq\emptyset$. If $M\cap E(G_2^{x_i})\neq\emptyset$ for some $x_i$, then we have the following Claim.

\textbf{Claim:} $M\cap E(G_2^{x_j})$ is a nice perfect matching of $G_2^{x_j}$ for each $x_j\in V(G_1)$, and  $\sigma_{x_j}^{-1}(M\cap E(G_2^{x_j}))=\sigma_{x_i}^{-1}(M\cap E(G_2^{x_i}))$. So $M\cap E(G_1^{v})=\emptyset$ for each $v\in V(G_2)$.
\begin{figure}[htbp!]
\centering
\includegraphics[height=2.6cm]{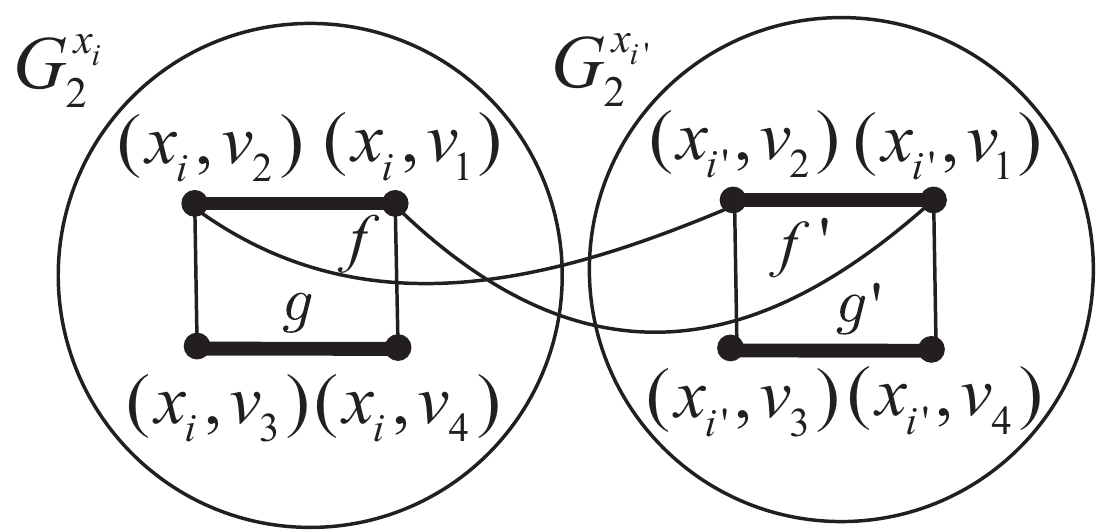}
\caption{\label{Fig.6}{\small Illustration for the proof of the Claim in Theorem \ref{cartesian  iff}.}}
\end{figure}

Take an edge $f=(x_i, v_1)(x_i, v_2)\in M\cap E(G_2^{x_i})$. Then $v_1v_2\in E(G_2)$.
If $n_2=2$, then $M\cap E(G_2^{x_i})=\{f\}$ is a nice perfect matching of $G_2^{x_i}$.
For $n_2\geq 3$, since $G_2$ is connected, without loss of generality we may assume that $d_{G_2}(v_2)\geq 2$.
Let $v_3$ be a neighbor of $v_2$ that is different from $v_1$.
So $(x_i, v_2)(x_i, v_3)\in E(G_2^{x_i})$.
Let $g$ be an edge of $M$ with an end-vertex $(x_i, v_3)$. Since $M$ is a nice perfect matching of $G$,
the other end-vertex of $g$ must be adjacent to $(x_i, v_1)$ by Theorem \ref{sufficiency and necessary}. So the other end-vertex of $g$ belongs to $V(G_2^{x_i})$ (see Fig. \ref{Fig.6}), that is, $g\in E(G_2^{x_i})$.
Since $G_2^{x_i}\cong G_2$ is a connected graph, we can obtain that $M\cap E(G_2^{x_i})$ is a perfect matching of $G_2^{x_i}$ in the above way. So $M\cap E(G_2^{x_i})$ is a nice perfect matching of $G_2^{x_i}$ by Proposition \ref{succession}.

Since $G_1$ is connected and $n_1\geq2$, there is some vertex $x_{i'}$ of $G_1$ such that $x_i$ and $x_{i'}$ are adjacent in $G_1$.
So vertex $(x_{i'}, v_1)\notin G_2^{x_i}$ is adjacent to $(x_i, v_1)$ in $G_1\Box G_2$ (see Fig. \ref{Fig.6}). Let $f'$ be an edge of $M$ that is incident with $(x_{i'}, v_1)$. Since $M$ is a nice perfect matching of $G_1\Box G_2$, the other end-vertex of $f'$ must be adjacent to $(x_i, v_2)$ by Theorem \ref{sufficiency and necessary}.
So $f'=(x_{i'}, v_1)(x_{i'}, v_2)\in M\cap E(G_2^{x_{i'}})$.
As the above proof, we can similarly show that $M\cap E(G_2^{x_{i'}})$ is a nice perfect matching of $G_2^{x_{i'}}$.
Since $G_1$ is connected, in an inductive way we can show that $M\cap E(G_2^{x_j})$ is a nice perfect matching of $G_2^{x_j}$ for any $x_j\in V(G_1)$.

Notice that $\sigma_{x_i}^{-1}(f)=v_1v_2=\sigma_{x_i'}^{-1}(f')$.
Let $g'$ be the edge of $M$ that is incident with $(x_{i'}, v_3)$.
Since $(x_{i'}, v_3)$ is adjacent to $(x_i, v_3)$, the other end vertex of $g'$ must be adjacent to the other end vertex $(x_{i}, v_4)$ of $g$ by Theorem \ref{sufficiency and necessary}.
So $g'=(x_{i'}, v_3)(x_{i'}, v_4)$ since $g'\in E(G_2^{x_{i'}})$.
This implies that $\sigma_{x_{i'}}^{-1}(g')=v_3v_4=\sigma_{x_i}^{-1}(g)$. In an inductive way, we can show that $\sigma_{x_{i'}}^{-1}(M\cap E(G_2^{x_{i'}}))=\sigma_{x_i}^{-1}(M\cap E(G_2^{x_i}))$.
Similarly, we also have $\sigma_{x_j}^{-1}(M\cap E(G_2^{x_j}))=\sigma_{x_i}^{-1}(M\cap E(G_2^{x_i}))$ for any $x_j\in V(G_1)$.

By this Claim, $M_2:=\sigma_{x_i}^{-1}({M\cap E(G_2^{x_i})})$ is a nice perfect matching of $G_2$ with $M=\sigma(M_2)$. If $M\cap E(G_1^{v_j})\neq\emptyset$, then we can similarly show that $G_1$ has a nice perfect matching $M_1$ with $M=\rho(M_1)$. So $\Phi^\ast(G_1\Box G_2)=\Phi^\ast(G_1)+\Phi^\ast(G_2)$.
\end{proof}

In fact, we can get the following corollary.
\begin{cor}\label{cartesian corollary}
Let $G$ be a simple connected graph. Then we have $\Phi^\ast(G)=\sum\limits_{i=1}^{k}\Phi^\ast(G_i)$ for any decomposition $G_1$$\Box\cdots$$\Box G_k$ of $G$.
\end{cor}

Now, it is easy to get the following proposition.
\begin{prop}\label{prime decomposition of G}
A simple connected graph $G$ is an extremal graph if and only if one of its Cartesian product factors is an extremal graph.
\end{prop}

The \emph{$n$-dimensional enhanced hypercube} $Q_{n, k}$, see \cite{1991}, is the graph with vertex set $V(Q_{n, k})$$=V(Q_n)$ and edge set $E(Q_{n, k})=E(Q_n)\cup \{(x_1x_2\cdots x_{n-1}x_n, \bar{x}_1\bar{x}_2\cdots\bar{x}_{n-k-1}$$\bar{x}_{n-k}$$x_{n-k+1}$
$x_{n-k+2}$$\cdots$$x_n: x_1x_2\cdots x_n\in V(Q_{n, k})\}$, where $0\leq k\leq n-1$. Clearly, $Q_n\cong Q_{n, n-1}$ and
$FQ_n\cong Q_{n, 0}$, i.e., the hypercube and the
folded hypercube are regarded as two special cases of the
enhanced hypercube.  By \cite{2015}, we have  $Q_{n, k}\cong FQ_{n-k}\Box Q_k$, for $0\leq k\leq n-1$. Hence we obtain the following result by the Proposition \ref{prime decomposition of G}.
\begin{cor}
$Q_{n, k}$ is an extremal graph and $Af(Q_{n, k})=n2^{n-2}$.
\end{cor}

According to the above discussion, for any graph $G$, we know that $K_{m, m}\Box G$, $K_{2n}\Box G$, $Q_n\Box G$,  $FQ_n\Box G$ and $Q_{n, k}\Box G$ are extremal graphs. Moreover, we can produce an infinite number of extremal graphs from an extremal graph by the Cartesian product operation.
%--------------------------------------------------------------------------------------------------------------------------------
\section{Further applications}
From examples we already know that $K_{m, m}$, $K_{2n}$, $Q_n$, $FQ_n$ and $Q_{n, k}$ are extremal graphs. Two perfect matchings $M_1$ and $M_2$ of a graph $G$ are called \emph{equivalent} if there is an automorphism $\varphi$ of $G$ such that $\varphi(M_1)=M_2$. So we know that all the perfect matchings of $K_{m, m}$ (or $K_{2n}$) are nice and equivalent. Further in this section we will count  nice perfect matchings of the three cube-like graphs.
\begin{thm}\label{n_nicepm}
$Q_n$ has exactly $n$ nice perfect matchings $E_1, E_2, \ldots, E_n$, all of which are equivalent.
\end{thm}
\begin{proof}
By Lemma \ref{Af(Q_n)}, $E_1, E_2, \ldots, E_n$ are $n$ distinct nice perfect matchings of $Q_n$. Since $Q_n$ is the Cartesian product of $n$ $K_2$'s, $Q_n$ has exactly $n$ nice perfect matchings by Corollary \ref{cartesian corollary}. So the first part is done. Now, it remains to show that $E_i$ and $E_j$ are equivalent for any $1\leq i<j\leq n$. Let the automorphism $f_{ij}$ of $Q_n$ be defined as $f_{ij}(x_1\cdots x_{i-1}x_ix_{i+1}\dots x_{j-1}x_jx_{j+1}$$\cdots x_n)=x_1\cdots x_{i-1}x_jx_{i+1}\dots x_{j-1}x_ix_{j+1}\cdots x_n$ for each vertex $x_1x_2\cdots x_n$ of $Q_n$. Then $f_{ij}(E_i)=E_j$.
\end{proof}

The theorem can be obtained  by applying the reversing-involutions of bipartite graphs, see \cite{GA}, but the computation is tedious.

Since  $FQ_2\cong K_4$ and $FQ_3\cong K_{4, 4}$, we have $\Phi^\ast(FQ_2)=3$ and $\Phi^\ast(FQ_3)=24$.
For $n\geq4$, we have a general result as follows.
\begin{thm}\label{nicepm_FQ_n}
$FQ_n$ has exactly $n+1$ nice perfect matchings for $n\geq4$.
\end{thm}
\begin{proof}
By Lemma \ref{Af(Q_n)}, $E_i$ is a perfect matching of $Q_n$. Then $E_i$ is also a perfect matching of $FQ_n$. We can easily check that $E_i$ is a nice perfect matching of $FQ_n$ by Theorem \ref{sufficiency and necessary}.

Let $E_{n+1}$ be the set of all the complementary edges of $FQ_n$. Then $E_{n+1}$ is a perfect matching of $FQ_n$. Let $u\bar{u}$ and $v\bar{v}$ be two distinct edges in $E_{n+1}$. Since any two distinct complementary edges are independent, the edge linked $u$ to $v$ or $\bar{v}$ (if exist) does not belong to $E_{n+1}$.
We can easily show that $uv\in E_j$ if and only if $\bar{u}\bar{v}\in E_j$ for some $j=1, 2, \ldots, n$, and $u\bar{v}\in E_s$ if and only if $\bar{u}v\in E_s$ for some $s=1, 2, \ldots, n$. So $E_{n+1}$ is also a nice perfect matching of $FQ_n$.

Now, we have found $n+1$ nice perfect matchings of $FQ_n$. Next, we will show that $FQ_n$ has no other nice perfect matchings.
By the contrary, we suppose that $M$ is a nice perfect matching of $FQ_n$ that is different from any $E_i$, $i=1, 2, \ldots, n+1$. Since $E_1, \ldots, E_{n+1}$ is a partition of the edge set $E(FQ_n)$, there is $E_k$ with $k\neq n+1$ such that $M\cap E_k\neq\emptyset$ and $E_k\neq M$. Clearly,  $FQ_n-(E_{n+1}\cup E_k)$ has exactly two components both of which are isomorphic to $Q_{n-1}$.
We notice that the $k$-th coordinate of each vertex in one component is $0$, and $1$ in the other component.
We denote the two components by $Q_n^{0}$ and $Q_n^{1}$, respectively.
In fact, $V(Q_n^{i})=\{x_1\cdots x_{k-1}ix_{k+1}\cdots x_n: x_j=0~\text{or}~1, j=1, \ldots, k-1, k+1, \ldots, n\}$, $i=0, 1$.
Since $M\cap E_{k}\neq\emptyset$, there is some edge $vv'\in M\cap E_k$ with $v\in V(Q_n^{0})$ and $v'\in V(Q_n^{1})$.
For any vertex $w$ of $Q_n^{0}$ with $w$ and $v$ being adjacent, we consider the edge $g$ of $M$ that is incident with $w$.
By Theorem \ref{sufficiency and necessary}, the other end-vertex of $g$ is adjacent to $v'$.
If $g=w\bar{w}$ is a complementary edge of $FQ_n$, then there are exactly two same bits in the strings of $\bar{w}$ and $v'$. So the edge $\bar{w}v'\in E(FQ_n)$ is not a complementary edge of $FQ_n$.
Since $\bar{w}$ and $v'$ are adjacent, there is exactly one different bit in the strings of $\bar{w}$ and $v'$.
So $n=3$, a contradiction. If $g=wz\in E(Q_n^{0})$, then there are exactly three different bits in the strings of $z$ and $v'$.
Since $z$ and $v'$ are adjacent in $FQ_n$, the edge $zv'$ is a complementary edge of $FQ_n$. So $n=3$, a contradiction. Hence $g\in E_k$. Since $Q_n^{0}$ is connected, using the above method repeatedly, we can show that $M=E_k$, a contradiction. So $FQ_n$ has exactly $n+1$ nice perfect matchings.
\end{proof}

\begin{prop}\label{FQ_n_equ_nicepm}
All the nice perfect matchings of $FQ_n$ $(n\geq2)$ are equivalent.
\end{prop}
\begin{proof}
We notice that $FQ_2\cong K_4$ and $FQ_3\cong K_{4, 4}$. So all the nice perfect matchings of $FQ_n$ are equivalent for $2\leq n\leq3$. Suppose that $n\geq4$.
From the proof of Theorem \ref{nicepm_FQ_n} we know that $E_1$, $E_2$, $\ldots,$ $E_{n+1}$ are all the nice perfect matchings of $FQ_n$. $f_{ij}$ defined in the proof of Theorem \ref{n_nicepm} is also an automorphism of $FQ_n$
such that $\varphi(E_i)=E_j$ for  $1\leq i<j\leq n$.
We will show that $E_1$ and $E_{n+1}$ are equivalent. Clearly, $FQ_n-(E_1\cup E_{n+1})$ has exactly two components each  isomorphic to $Q_{n-1}$, denoted by $Q_n^{0}$ and $Q_n^{1}$. Set $V(Q_n^i)=\{ix_2x_3\cdots x_n: x_j=0~\text{or}~1, j=2, \ldots, n\}$, $i=0, 1$.
We define a bijection $f$ on $V(FQ_n)$ as follows:
\begin{equation}
f(x_1x_2\cdots x_n)=
\begin{cases}
\bar{x}_1x_2\cdots x_n, & \text{if $x_1x_2\cdots x_n\in V(Q_n^{0})$,}\nonumber\\
\bar{x}_1\bar{x}_2\cdots \bar{x}_n, & \text{if $x_1x_2\cdots x_n\in V(Q_n^{1})$}.
\end{cases}
\end{equation}
It is easy to check that $f$ is an automorphism of $FQ_n$. In addition, $f(E_1)=E_{n+1}$. Hence all the nice perfect matchings of $FQ_n$ are equivalent.
\end{proof}

By Corollary \ref{cartesian corollary} and Theorems \ref{n_nicepm} and \ref{nicepm_FQ_n}, we can  obtain the following conclusion.
\begin{cor}
$\Phi^\ast(Q_{n, n-1})=n$,
$\Phi^\ast(Q_{n, n-2})=n+1$, $\Phi^\ast(Q_{n, n-3})=n+21$ and $\Phi^\ast(Q_{n, k})=n+1$ for any $0\leq k\leq n-4$.
\end{cor}
\begin{prop}
For $0<k<n-1$, $Q_{n, k}$ has exactly two nice perfect matchings up to the equivalent.
\end{prop}
\begin{proof}
Since $Q_{n, k}=FQ_{n-k}\Box Q_k$, by adapting the notations in Eq. (\ref{map}) and by the proof of Theorem \ref{cartesian  iff} we know that
all the nice perfect matchings of $Q_{n, k}$ are divided into two classes $\mathcal{M}'$ and $\mathcal{M}''$, where $\mathcal{M}'=\{\rho(M): M$ is a nice perfect matching of $FQ_{n-k}\}$ and $\mathcal{M}''=\{\sigma(M): M$ is a nice perfect matching of $Q_k\}$.

For $M_1'$, $M_2'\in\mathcal{M}'$, there are two nice perfect matchings $M_1$ and $M_2$ of $FQ_{n-k}$ such that $M_i'=\rho(M_i)$, $i=1, 2$.
By Proposition \ref{FQ_n_equ_nicepm}, there exists an automorphism $\varphi$ of $FQ_{n-k}$ such that $\varphi(M_1)=M_2$. Let $\varphi'(x, u):=(\varphi(x), u)$ for each vertex $(x, u)$ of $FQ_{n-k}\Box Q_k$.
It is easy to check that $\varphi'$ is an automorphism of $Q_{n, k}$ and $\varphi'(M_1')=M_2'$. By the arbitrariness of $M_1'$ and $M_2'$, we know that all the nice perfect matchings in $\mathcal{M}'$ are equivalent.
Similarly, we can show that all the nice perfect matchings in $\mathcal{M}''$ are equivalent.
\begin{figure}[htbp!]
\centering
\includegraphics[height=3.4cm]{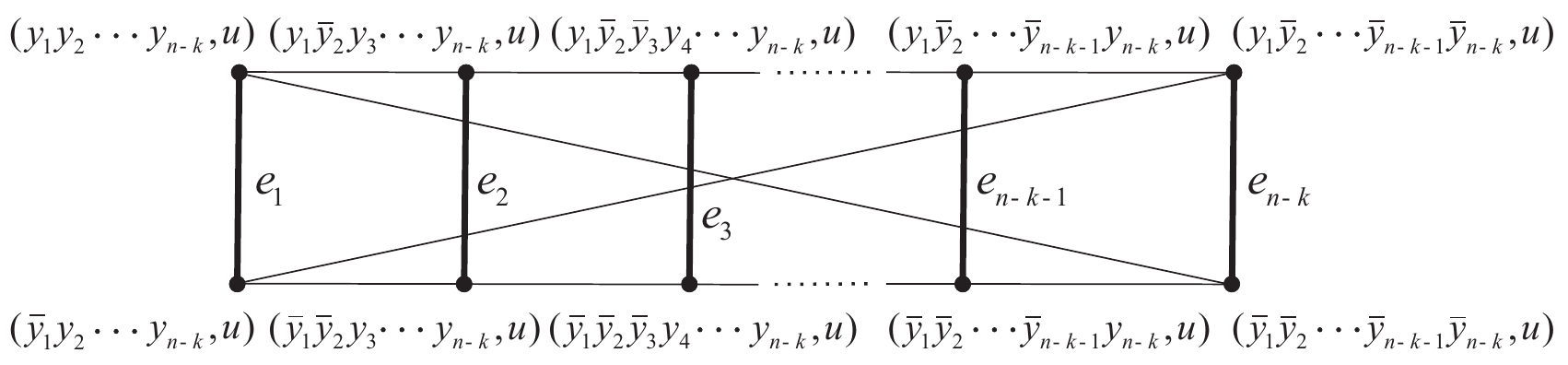}
\caption{\label{Fig.7}{\small The graph $H$.}}
\end{figure}

Let $F_1$ and $E_1$ be the sets of all the $1$-edges of $FQ_{n-k}$ and $Q_k$ respectively. Then $F_1$ is a nice perfect matching of $FQ_{n-k}$ and $E_1$ is a nice perfect matching of $Q_k$. So $\rho(F_1)\in \mathcal{M}'$ and $\sigma(E_1)\in \mathcal{M}''$.
See Fig. \ref{Fig.7}, we choose a subset $S:=\{e_1, \ldots, e_{n-k}\}$ of $\rho(F_1)$. Then all the vertices incident with $S$ induce a subgraph $H$ as depicted in Fig. \ref{Fig.7}.
For any subset $R\subseteq \sigma(E_1)$ of size $n-k$, let $G$ be the subgraph of $Q_{n, k}$ induced by all the vertices incident with $R$.
We note that $Q_{n,k}-\sigma(E_1)$ has exactly two components $A$ and $B$ each of which is isomorphic to $FQ_{n-k}\Box Q_{k-1}$, and $\sigma(E_1)=E(A, B)$.
So $G-R$ has at least two components. Clearly $H-S$ is connected. So for any automorphism  $\psi$ of $Q_{n, k}$, $\psi(S)\neq R$.
By the arbitrariness of $R$ we know that $\rho(F_1)$ and $\sigma(E_1)$ are not equivalent. Then we are done.
\end{proof}

From Corollary \ref{cartesian corollary} it is helpful to give a Cartesian decomposition of an extremal graph. It is known that  $Q_n\cong K_2\Box\cdots \Box K_2$ and $Q_{n,k}\cong FQ_{n-k}\Box Q_k$. However we shall see  surprisedly  that $FQ_n$ is undecomposable.

A nontrivial graph $G$ is said to be \emph{prime} with respect to the Cartesian product if whenever $G\cong H\Box R$, one factor is isomorphic to the complete graph $K_1$ and the other is isomorphic to $G$. Clearly, for $m\geq3$ and $n\geq2$, $K_{m, m}$ and $K_{2n}$ are prime extremal graphs. In the sequel, we show that $FQ_n$ is a prime extremal graph, too.

Recall that the length of a shortest path between two vertices $x$ and $y$ of $G$ is called the distance between $x$ and $y$, denoted by $d_G(x, y)$.
Let $G$ be a connected graph. Two edges $e=xy$ and $f=uv$ are in the relation $\Theta_G$ if
$d_G(x, u)+d_G(y, v)\neq d_G(x, v)+d_G(y, u).$
Notice that $\Theta_G$ is reflexive and symmetric, but need not to be transitive. We denote its transitive closure by $\Theta_G^{\ast}$. For an even cycle $C_{2n}$, $\Theta_{C_{2n}}$ consists of all pairs of antipodal edges. Hence, $\Theta_{C_{2n}}^{\ast}$ has $n$ equivalence classes and $\Theta_{C_{2n}}=\Theta_{C_{2n}}^{\ast}$. For an odd cycle $C$, any edge of $C$ is in relation $\Theta$ with its two antipodal edges. So all edges of $C$ belong to an equivalence class with respect to $\Theta^{\ast}_C$.
By the Cartesian product decomposition Algorithm depicted in \cite{WS}, we have the following lemma.
\begin{lem}\label{cartesian product decomposition algorithm}
If all the edges of a graph $G$ belong to an equivalence class with respect to $\Theta_G^{\ast}$, then $G$ is a prime graph under the Cartesian product.
\end{lem}

The \emph{Hamming distance} between two vertices $x$ and $y$ in $Q_n$ is the number of different bits in the strings of both vertices, denoted by $H_{Q_n}(x, y)$ .
\begin{thm}[\cite{XUJUNM}]\label{property of FQ_n}
For a folded hypercube $FQ_n$, we have

$(1)$ $FQ_n$ is a bipartite graph if and only if $n$ is odd.

$(2)$ The length of any cycle in $FQ_n$ that contains exactly one complementary edge is at least $n+1$. If $n$ is even, then the length of a shortest odd cycle in $FQ_n$ is $n+1$.

$(3)$ Let $u$ and $v$ be two vertices in $FQ_n$. If $H_{Q_n}(u, v)\leq\lfloor\frac{n}{2}\rfloor$, then any shortest $uv$-path
in $FQ_n$ contains no complementary edges. If $H_{Q_n}(u, v)>\lceil\frac{n}{2}\rceil$, then any shortest $uv$-path in $FQ_n$ contains exactly one complementary edge.
\end{thm}

Here we list some known properties of $Q_n$ that will be used in the sequel. For any two vertices $x$ and $y$ in $Q_n$, $d_{Q_n}(x, y)=H_{Q_n}(x, y)$.
For any shortest path $P$ from $x_1x_2\cdots x_n$ to $\bar{x}_1\bar{x}_2\cdots\bar{x}_n$ in $Q_n$, $|E(P)\cap E_i|=1$ for each $i=1, 2, \ldots, n$.
For any integer $j$ ($1\leq j\leq n$), there is a shortest path $P$ from $x_1x_2\cdots x_n$ to $\bar{x}_1\bar{x}_2\cdots\bar{x}_n$ in $Q_n$ such that the edge in $E(P)\cap E_i$ is the $j$th edge when traverse $P$ from $x_1x_2\cdots x_n$ to $\bar{x}_1\bar{x}_2\cdots\bar{x}_n$.

For every subgraph $F$ of a graph $G$, the inequality $d_F(u, v)\geq d_G(u, v)$ obviously holds. If $d_F(u, v)= d_G(u, v)$ for all $u, v\in V(F)$, we say $F$ is an \emph{isometric subgraph} of $G$.
\begin{prop}[\cite{Handbook}]\label{isometric}
Let $C$ be a shortest cycle of $G$. Then $C$
is isometric in $G$.
\end{prop}
\begin{thm}
$FQ_n$ is a prime graph under the Cartesian product.
\end{thm}
\begin{proof}
Clearly, $FQ_2$ and $FQ_3$ are prime. So we suppose that $n\geq4$. We recall that $E_i$ is the set of all the $i$-edges of $Q_n$, $i=1, 2,$ $\ldots$, $n$. Let $E_{n+1}$ be the set of all the complementary edges of $FQ_n$. Then $E_1$, $E_2$, $\ldots$, $E_{n+1}$ is a partition of $E(FQ_n)$. Since the girth of $FQ_n$ is $4$ for $n\geq4$, any two opposite edges of a $4$-cycle are in relation $\Theta_{FQ_n}$. So $E_i$ is contained in an equivalence class with respect to $\Theta_{FQ_n}^\ast$, $i=1, 2, \ldots, n+1$.
For any vertex $x_1x_2\cdots x_n$, it is linked to $\bar{x}_1\bar{x}_2\cdots\bar{x}_n$ by a complementary edge $e$ in $FQ_n$. Let $P$ be any shortest path from $x_1x_2\cdots x_n$ to $\bar{x}_1\bar{x}_2\cdots\bar{x}_n$ in $Q_n$. Then the length of $P$ is $n$ and $|P\cap E_i|=1$ for any $i=1, 2, \ldots, n$.
Set $C:=P\cup \{e\}$. Then $C$ is a cycle of length $n+1$.

If $n$ is even, then the length of any shortest odd cycle in $FQ_n$ is $n+1$ by Theorem \ref{property of FQ_n} $(2)$. So $C$ is a shortest odd cycle in $FQ_n$. By Proposition \ref{isometric}, $C$ is an isometric odd cycle in $FQ_n$.
So all edges of $C$ belong to an equivalence class with respect to $\Theta_{FQ_n}^\ast$. Since $E(C)\cap E_i\neq\emptyset$ for any $i=1, 2, \ldots, n+1$,  all edges of $E(FQ_n)=\bigcup\limits_{i=1}^{n+1}E_i$ belong to an equivalence class with respect to $\Theta_{FQ_n}^\ast$, that is, $FQ_n$ is a prime graph under the Cartesian product by Lemma \ref{cartesian product decomposition algorithm}.

For $n$ being odd, we first show that $C$ is an isometric cycle in $FQ_n$. It is sufficient to show that $d_C(u, v)=d_{FQ_n}(u, v)$ for any two distinct vertices $u$ and $v$ of $C$. By Theorem \ref{property of FQ_n} $(3)$, there are two cases for the shortest $uv$-path in $FQ_n$. If $H_{Q_n}(u, v)\leq\lfloor\frac{n}{2}\rfloor$, then any shortest $uv$-path
in $FQ_n$ contains no complementary edges. So $d_{FQ_n}(u, v)=d_{Q_n}(u, v)=H_{Q_n}(u, v)=d_C(u, v)$.
If $H_{Q_n}(u, v)>\lceil\frac{n}{2}\rceil$, then any shortest $uv$-path in $FQ_n$ contains exactly one complementary edge. Let $P_1$ be the $uv$-path on $C$ that contains the unique complementary edge $e$. Since $H_{Q_n}(u, v)>\lceil\frac{n}{2}\rceil$ and the length of $C$ is $n+1$, $d_C(u, v)=|P_1|=n+1-H_{Q_n}(u, v)<\lceil\frac{n}{2}\rceil$. Clearly $d_{FQ_n}(u, v)\leq d_C(u, v)$.
We suppose that $d_{FQ_n}(u, v)< d_C(u, v)$, that is, $P_1$ is not a shortest $uv$-path in $FQ_n$. Let $P_2$ be a shortest $uv$-path in $FQ_n$. Then $P_2$ contains exactly one complementary edge by Theorem \ref{property of FQ_n} $(3)$.
Set $P':=C-(V(P_1)\setminus\{u, v\})$. Then $P'\cup P_2$ is a walk in $FQ_n$ that has exactly one complementary edge. So there is a cycle $C'\subseteq P'\cup P_2$ that contains exactly one complementary edge. We can deduce a contradiction by Theorem \ref{property of FQ_n} $(2)$ as follows:
$$n+1\leq|C'|\leq|P'|+|P_2|<|P'|+|P_1|=|C|=n+1.$$
So $d_{FQ_n}(u, v)=d_C(u, v)$.

For any $i\in\{1, 2, \ldots, n\}$,
let $P^i$ be a shortest path from $x_1x_2\cdots x_n$ to $\bar{x}_1\bar{x}_2\cdots\bar{x}_n$ in $Q_n$ such that the unique edge in $P^i\cap E_i$ is the antipodal edge of $e$ on $C^i:=P^i\cup \{e\}$.
Since $C^i$ is an isometric even cycle by the above proof, the unique complementary edge $e$ on $C^i$ and its antipodal edge $P^i\cap E_i$ are in relation $\Theta_{FQ_n}$.  So $E_i$ and $E_{n+1}$ are contained in an equivalence class with respect to $\Theta_{FQ_n}^\ast$, $i=1, 2, \ldots, n$. Hence $FQ_n$ is a prime graph under the Cartesian product by Lemma \ref{cartesian product decomposition algorithm}.
\end{proof}

Now we know that for $m\geq3$ and $n\geq2$, $K_{m, m}$, $K_{2n}$  and $FQ_n$  are prime extremal graphs. From Proposition \ref{prime decomposition of G}, it is interesting to characterize all the prime extremal graphs.

\acknowledgements
We thank two anonymous reviewers for giving helpful suggestions and  comments to improve the manuscript.

%--------------------------------------------------------------------------------------------------------------------------------
\nocite{*}
\bibliographystyle{abbrvnat}
\bibliography{tight}
\end{document}